\documentclass[12pt]{amsart}
\usepackage{amssymb}
\usepackage{amsmath, amscd}
\usepackage{amsthm}
\usepackage{xcolor}
\usepackage{ulem}

\newtheorem{theorem}{Theorem}[section]

\newtheorem{proposition}[theorem]{Proposition}
\newtheorem{lemma}[theorem]{Lemma}

\newtheorem{corollary}[theorem]{Corollary}
\theoremstyle{definition}

\newtheorem{example}[theorem]{Example}
\newtheorem{definition}[theorem]{Definition}

\newtheorem{problem}[theorem]{Problem}

\newcommand{\N}{\mathbb{N}}
\newcommand{\Z}{\mathbb{Z}}
\newcommand{\Q}{\mathbb{Q}}

\newcommand{\qc}{\Z(p^\infty)}

\newcommand{\Ker}{{\rm Ker}}

\newcommand{\s}{\sigma}

\def\a{\alpha}
\def\b{\beta}

\def\o{\omega}

\def\s{\sigma}

\def\val#1{\vert #1 \vert}

\usepackage[utf8]{inputenc}
\usepackage{amssymb}
\usepackage{amsmath, amscd}
\usepackage{amsthm}
\usepackage{xcolor}

\begin{document}
	
\author[A.R. Chekhlov]{Andrey R. Chekhlov}
\address{Department of Mathematics and Mechanics, Tomsk State University, 634050 Tomsk, Russia}
\email{cheklov@math.tsu.ru; a.r.che@yandex.ru}
\author[P.V. Danchev]{Peter V. Danchev}
\address{Institute of Mathematics and Informatics, Bulgarian Academy of Sciences, 1113 Sofia, Bulgaria}
\email{danchev@math.bas.bg; pvdanchev@yahoo.com}
\author[B. Goldsmith]{Brendan Goldsmith}
\address{Technological University, Dublin, Dublin 7, Ireland}
\email{brendan.goldsmith@TUDublin.ie; brendangoldsmith49@gmail.com}
\author[P.W. Keef]{Patrick W. Keef}
\address{Department of Mathematics, Whitman College, Walla Walla, WA 99362, USA}
\email{keef@whitman.edu}
	
\title[Relatively and Weakly Hopfian Abelian Groups] {Two Generalizations of Hopfian Abelian Groups}
	
\maketitle
	
\begin{abstract} This paper targets to generalize the notion of Hopfian groups in the commutative case by defining the so-called {\it relatively Hopfian groups} and {\it weakly Hopfian groups}, and establishing some their crucial properties and characterizations. Specifically, we prove that for a reduced Abelian $p$-group $G$ such that $p^{\omega}G$ is Hopfian (in particular, is finite), the notions of relative Hopficity and ordinary Hopficity do coincide. We also show that if $G$ is a reduced Abelian $p$-group such that $p^{\omega}G$ is bounded and $G/p^{\omega}G$ is Hopfian, then $G$ is relatively Hopfian. This allows us to construct a reduced relatively Hopfian Abelian $p$-group $G$ with $p^{\o}G$ an infinite elementary group such that $G$ is {\it not} Hopfian.
	
In contrast, for reduced torsion-free groups, we establish that the relative and ordinary Hopficity are equivalent. Moreover, the mixed case is explored as well, showing that the structure of both relatively and weakly Hopfian groups can be quite complicated.
\end{abstract}
	
\vskip2.0pc
	
\section{Introduction and Basic Definitions}
	
All groups considered in the present work will be Abelian and additively written. Our classical notation and terminology will be in agreement with those from \cite{F1,F} and \cite{K}, respectively, and the new notions will be explicitly defined at appropriate places.

\medskip
	
In an important paper \cite{BP}, Beaumont and Pierce initiated a detailed study of groups that they called {\it ID-groups}, i.e., groups $G$ such that $G$ has an isomorphic proper direct summand. Another class of groups considered somewhat earlier was the class consisting of groups $G$ possessing a non-trivial subgroup $N$ such that $G \cong G/N$. In both cases, \lq most\rq \ infinite groups - in a not too precisely defined sense - have both of these properties. Thus, the study of groups failing to have these properties became of some interest.

\medskip
	
In another vein, in the early 1940s, Baer \cite{B} had begun a study of groups ({\it not} necessarily Abelian) which do {\it not} possess a proper isomorphic quotient. In modern terminology, the groups studied by Baer are called {\it Hopfian groups} - this terminology derives from the work of H. Hopf, who showed that the defining property holds in certain types of manifolds; the negation of the ID-groups is usually referred to as the {\it directly finite groups}, i.e., groups which do {\it not} possess an isomorphic proper direct summand. We will denote these classes by $\mathcal{H}$ and $\mathcal{DF}$, respectively.

\medskip
	
It is straightforward to show that a Hopfian group is necessarily directly finite, so $\mathcal{H} \subseteq \mathcal{DF}$ holds. In fact, the containment is even strict, that is, $\mathcal{H} \subsetneqq \mathcal{DF}$: indeed, since an indecomposable group is certainly directly finite, it suffices to find such a group which is {\it not} Hopfian  -- in this vein, Corner in \cite[Example 1]{C2} constructed such a non-Hopfian group in answer to a question posed about the Hopficity or, otherwise, of a group with automorphism group of order $2$. The directly finite groups are {\it not} the only natural generalizations of Hopficity. The condition that a group possesses a non-trivial subgroup $N$ with $G \cong G/N$ can be strengthened to ask that the subgroup $N$ be {\it pure} in $G$ and the negation of this class leads to a possibly larger class of groups than the Hopfian class as we will demonstrate in the sequel.

\medskip

It is worthwhile noticing that some important results in the general theory of Hopfian groups ({\it not} necessarily commutative) were obtained in \cite{H1} and \cite{H2}.
	
\medskip
	
The present paper is concerned with the following two further variations on these ideas.
	
\begin{definition}\label{newone}
The group $G$ is said to be \textbf{relatively Hopfian} if whenever $\phi: G\to  G\oplus A$ is a surjective homomorphism for some group $A$, then we must have $A=\{0\}$. We denote the class of relatively Hopfian groups by $\mathcal {RH}$.
\end{definition}
	
A quick trick shows that this definition is just equivalent to saying that {\bf $G$ is not isomorphic to a proper direct summand of the factor-group $G/H$ for any non-trivial subgroup $H$ of $G$}.
	
\medskip
	
In fact, if $G$ fails in the above definition, then there is a surjection $\pi: G\to G\oplus A$ with $A\ne \{0\}$. So, there is also a surjection $\gamma:G\to G\oplus A_1\oplus A_2$, where $A_1$ and $A_2$ are copies of $A$. If we set $H:=\gamma^{-1}(A_2) \ne \{0\}$, then one sees that $G/H\cong G\oplus A_1$, so our claim also fails. Reciprocally, if our claim fails, then there is a surjection $G\to G/H\cong G\oplus A$, where $H, A\ne \{0\}$, so our definition also fails.
	
\medskip
	
We, however, emphasize that the approach proposed in Definition~\ref{newone} is easier, since it only involves one other group, $A$, and thus makes it clear that it is {\it not} relevant whether $H\ne \{0\}$.
	
\begin{definition}
The group $G$ is said to be \textbf{weakly Hopfian} if whenever $\pi: G\to G$ is a surjective homomorphism whose kernel is a pure subgroup $J$ of $G$, then we must have $J=\{0\}$. We denote the class of weakly Hopfian groups by $\mathcal {WH}$.
\end{definition}
	
Adapting the same method, as illustrated above, this definition just says that $G\not\cong G/N$ for any non-trivial pure subgroup $N\leq G$.
	
\medskip
	
Henceforth, our plan of work is organized thus: in Section~\ref{PR}, we establish some basic properties of these definitions. In particular (Theorem~\ref{containments}), we verify that
	\[
	\mathcal{H}\subseteq \mathcal{RH}\subseteq \mathcal{WH}\subseteq  \mathcal {DF}.
	\]
Other basic properties are also established, such as the closure of these classes with respect to summands (Proposition~\ref{prop3}).
	
If $\mathcal A$ is a class of groups, we will denote the class of torsion and torsion-free groups in $\mathcal A$ by $\mathbf T\mathcal A$ and $\mathbf F\mathcal A$, respectively. Since it is almost clear, in virtue of what we established below, that a torsion-free group $G$ is weakly Hopfian if, and only if, it is Hopfian, we can conclude that
	\[
	\mathbf F \mathcal{H}= \mathbf F \mathcal{RH}= \mathbf F \mathcal{WH}
	\]
	(Theorem~\ref{suggestion}).
	
	In Section~\ref{tsect}, we consider the case of torsion groups. In particular, we show that
	\[
	\mathbf T \mathcal{H}\subsetneq \mathbf T \mathcal{RH}\subsetneq \mathbf T \mathcal{WH}= \mathbf T   \mathcal {DF}
	\]
Regarding the first inclusion, it is relatively easy to see that, for any prime $p$, the infinite co-cyclic $p$-group $\Z(p^\infty)$ is relatively Hopfian, but {\it not} Hopfian. More generally, we show that a $p$-group $G$ is relatively Hopfian if, and only if, whenever $H$ is a subgroup of $G$ such that $G/H\cong G$, then $H$ must be a bounded subgroup of $p^\omega G$ (Theorem~\ref{bigresult}). This implies, for example, that a separable $p$-group is relatively Hopfian if, and only if, it is Hopfian (Corollary~\ref{ulmhopfian}). On the other hand, using a combination of results of Pierce and Corner, we are able to construct a (reduced) $p^{\o+1}$-bounded $p$-group that is relatively Hopfian, but {\it not} Hopfian (Example~\ref{piercecorner}).

\medskip
	
We show that $p$-group $G$ is in $\mathbf T \mathcal{WH}= \mathbf T   \mathcal {DF}$ if, and only if, the Ulm invariant $f_\a (G)$ is finite whenever $\a\in \o^\infty=\{0, 1, 2, \dots\}\cup\{\infty\}$ (Theorem~\ref{torsionweak}). This means that there are lots of examples of weakly Hopfian groups that are {\it not} relatively Hopfian, e.g., the standard direct sum of cyclic $p$-groups $B=\oplus_{n\in \N} ~ \Z(p^n)$.

\medskip
	
Furthermore, in Section~\ref{mixedgroups}, we consider the case of mixed groups. Our main result concerns the case where $T$ is the torsion subgroup of $G$ and $G/T$ is Hopfian (equivalently, relatively Hopfian, or weakly Hopfian). We show that if $T$ is in any one of the classes $\mathcal H$, $\mathcal {RH}$, $\mathcal {WH}$ or $\mathcal {DF}$, then $G$ is in the same class too (Corollary~\ref{quotients}). A few other results of different branches are also established (see, e.g., Propositions~\ref{sums} and \ref{cotorsion}).

\medskip
	
Finally, in Section 5, we close our examination with two intriguing questions, namely Problems 5.1 and 5.2, that might be quite difficult.
		
\section{Preliminary Results}\label{PR}
	
Our first result is simple but useful.
	
\begin{proposition}\label{prop3} The classes $\mathcal H$, $\mathcal {RH}$, $\mathcal {WH}$ and $\mathcal {DF}$ are closed with respect to direct summands.
\end{proposition}
	
\begin{proof}  Since all of these arguments are very similar, we will be content with showing $\mathcal {RH}$ has the indicated property. Let $G$ be relatively Hopfian and write $G = H \oplus K$. Suppose we have a surjective homomorphism $\lambda:K\to K\oplus A$; we need to show $A=\{0\}$. Therefore, $\phi:=(1_H, \lambda):G \to G\oplus A$ will also be surjective, and since $G$ is relatively Hopfian, we can conclude $A=\{0\}$, as desired.
\end{proof}
	
We continue this section with a technical result. For a prime $p$, we will denote the $p$-height function on a group $G$ by $\mid - \mid^{p}_G$. So, if $H$ is a subgroup of $G$, $$\min_G \mid H \mid^{p}_G:=\min\{\mid x \mid^{p}_G : x\in H\}.$$ If there is no danger of ambiguity, we simplify this notation to $\val -$ and $\min \val H$, respectively.

\begin{lemma}\label{hzero}
Suppose that $G$ is an arbitrary group, $p$ is a prime and $H$ is a subgroup of $G$. If $G/H\cong G$ and $\min \val H=0$, then $G$ is not relatively Hopfian.
\end{lemma}
	
\begin{proof} Suppose $h\in H$ satisfies $\val h=0$. Let $$N=H\cap pG=\{x\in H: \val x_G\geq 1\};$$ so, $h+N$ is a non-zero element of $A:=H/N$. Note that if $x\in H$, then $px\in N$, so that $A$ is a non-zero $p$-bounded subgroup of $G/N$. In addition, for every $\overline{0}\ne x+N\in A$, if $x+N=p(y+N)$ for some $y+N\in G/N$, then $x-py\in N\subseteq pG$. This would imply that $x-py=pz$ for some $z\in G$, and hence $x=p(y+z)$. However, this in turn would imply that $x\in H\cap pG=N$, contrary to assumption.
		
It follows that for all non-zero $x+N\in A$, that
		\[
		0\leq \val {x+N}_{A}\leq \val {x+N}_{G/N} =0.
		\]
Therefore, $A = H/N$ is a $p$-pure, $p$-subgroup of $G/N$ and so $A$ must be pure in $G/N$. But since $A$ is $p$-bounded, this implies that $G/N\cong X\oplus A$, where $$X\cong (G/N)/A=(G/N)/(H/N)\cong G/H\cong G.$$ In other words, there is a surjective homomorphism $G\to G\oplus A$ and $G$ is not relatively Hopfian, as stated.
\end{proof}
	
\begin{theorem}\label{containments} We have the following containments:
		\[
		\mathcal{H}\subseteq \mathcal{RH}\subseteq \mathcal{WH}\subseteq  \mathcal {DF}.
		\]
\end{theorem}
	
\begin{proof} We do all of these arguments indirectly. If $G$ is not directly finite, then there is an isomorphism $G\cong G\oplus A$ with $A\ne \{0\}$. Therefore, the usual projection $G\cong G\oplus A\to G$ will be surjective with pure kernel isomorphic to $A\ne \{0\}$, so that $G$ is not weakly Hopfian.
		
If $G$ is not relatively Hopfian then let $\phi:G\to G\oplus A$ be surjective with $A\ne \{0\}$. Again, if $\pi: G\oplus A\to G$ is the usual projection, then $\pi\circ \phi:G\to G$ is a surjective endomorphism with a non-zero kernel, showing that $G$ is not Hopfian.
		
Finally, suppose $G$ is not weakly Hopfian. We want to show it is not relatively Hopfian.
By definition, there is a non-zero pure subgroup $H$ of $G$ with $G \cong G/H$. If $H$ were divisible, we could conclude that $$G\cong (G/H)\oplus H\cong G\oplus H,$$ so $G$ will not be relatively Hopfian, as desired.
		
On the other hand, suppose $H$ \textit{fails} to be divisible. So, for some prime $p$, there is an $x\in H$ such that $0=\val x_H=\val x_G$. So, Lemma~\ref{hzero} applies, showing that $G$ is not relatively Hopfian, as wanted.
\end{proof}
	
As noted in the introductory section, since for torsion-free groups, being Hopfian agrees with being weakly Hopfian (see Theorem~\ref{suggestion} quoted below), we have the following consequence of the last result:
	
\begin{corollary}\label{cor2.4} The following equalities hold:
		\[
		\mathbf F \mathcal{H}= \mathbf F \mathcal{RH}= \mathbf F \mathcal{WH}.
		\]
\end{corollary}
	
\bigskip
	
Clearly, if $G$ has a direct summand of the form $A^{(\o)}$ for some $A\ne \{0\}$, then $G\cong G\oplus A$, so that it is not directly finite. We will frequently use the consequent fact that no group in any of the four classes discussed in the Introduction can have such a direct summand. In particular, a directly finite group with non-trivial torsion must be $p$-semi-standard for all primes $p$, where we define the notion of $p$-semi-standard as follows:
	
\medskip
	
Let $\mathcal P$ be the collection of all primes. If $p\in \mathcal P$, we will say the group $G$ is $p$-\textit{semi-standard} if, for all $n<\omega$, the Ulm-Kaplansky invariant $$f_n(G)={\rm rank}((p^nG)[p]/(p^{n+1}G)[p])$$ is finite.
	
\medskip
	
As an other consequence of interest, we obtain:
	
\begin{corollary}\label{cor3} Suppose $G=R\oplus D$, where $R$ is reduced and $D$ is divisible. Assume further that $D=D_0\oplus (\bigoplus_{p\in \mathcal P}D_p)$, where $D_0$ is torsion-free and $D_p$ is a $p$-group. If $G$ is directly finite, then $D_0$ has finite rank and, for all $p\in \mathcal P$, $D_p$ has finite $p$-rank and $R$ is $p$-semi-standard.
\end{corollary}
	
\begin{proof}
Clearly, if $D_0$ has infinite rank or some $D_p$ has infinite $p$-rank, then $G$ will have a summand of the form $A^{(\o)}$, where either $A$ is isomorphic to $\Q$ or to some $\Z(p^\infty)$. In addition, if $p\in \mathcal P$ and $n<\o$ is such that $f_n(G)=f_n(R)$ is infinite, and $B$ is a $p$-basic subgroup of $R$, then $B$ will have a direct summand isomorphic to $\Z(p^{n+1})^{(\o)}$. Since this summand will be pure and bounded in $G$, it will also be a direct summand of $G$. In any of the above cases, then, $G$ will fail to be directly finite, completing the arguments.
\end{proof}
	
Again, the above observation also applies to the other three classes we are considering. 	
	
\begin{theorem}\label{prop218} Suppose $K$ and $L$ are groups such that ${\rm Hom}(K,L)=\{0\}$. Then, $G=L\oplus K$ is in one of $\mathcal H$, $\mathcal {RH}$, $\mathcal {WH}$ or $\mathcal {DF}$ if, and only if, $L$ and $K$ are both members of the same class.
\end{theorem}
	
\begin{proof} Certainly, if $G$ is in one of these classes, then so are $L$ and $K$. The converse will require individual arguments like these.
		
Suppose first that $L,K\in \mathcal {H}$. Let $\pi:L\oplus K\to L\oplus K$ be a surjective homomorphism with  kernel of $J$; we want to show $J=\{0\}$. Our hypothesis assures that $\pi(K)\subseteq K$. Therefore, $\pi$ induces a surjective composite homomorphism $$\pi':L\cong G/K\to G/\pi(K) \to G/K\cong L.$$ Since $L$ is Hopfian, we can conclude that this composite is injective. This readily implies that $J\subseteq K$, and $\pi(K)=K$. But then, since $K$ is Hopfian, we can now conclude that $J=\{0\}$, as desired.
		
Suppose now that $L,K\in \mathcal {WH}$. Again, as above, suppose  $\pi:L\oplus K\to L\oplus K$ is a surjective homomorphism with \textit{pure} kernel $J$; we still want to show $J=\{0\}$. We can again form the composite $$\pi': G/K\to G/\pi(K) \to G/K,$$ but in this case we will need to show that its kernel, which we denote by $J'$, is pure in $G/K\cong L$. Suppose, therefore, that $x+K\in G/K$ and $n x+K\in J'$; we want to find $x'+K\in J'$ such that $n x'+K=n x+K$. Note that we are assuming that $n \pi(x)=\pi(nx)\in K$. Suppose $\pi(x)=y+z$, where $y\in K$, $z\in L$. Since $n\pi(x)\in K$, we can conclude that $nz=0$. The purity of $J$ and the surjectivity of $\pi$ guarantee that we can find $w\in K\oplus L$ such that $\pi(w)=z$ and $nw=0$ as $nw\in J$. Letting $x':=x-w$, it follows that $\pi(x')=y+z-z=y\in K$, so that $x'+J$ is in the kernel of $\pi'$. And since $nw=0$, one finds that $nx'=nx$, so that $n x'+K=n x+K$, as asked for.
		
The rest of the argument goes through as before. Since $L$ is weakly Hopfian, we can conclude that $\pi'$ is injective, $J\subseteq K$, $\pi(K)=K$, and since $J$ will also be pure in $K$, it follows that $J=\{0\}$, as desired.
		
Suppose next that $L,K\in \mathcal {RH}$. Arguing indirectly, if $G$ is not relatively Hopfian, then there is group $A\ne \{0\}$ and a surjective homomorphism
		$\phi:G\to G\oplus A$. Thus, $\phi(K)\subseteq K\oplus A$. Since $K$ is relatively Hopfian, we cannot have $\phi(K)=K\oplus A$. So, if $A'=(K\oplus A)/\phi(K)$, then $A'\ne \{0\}$. Since there is a surjective composition
		$$
		L\cong (L\oplus K)/K \to (L\oplus K\oplus A)/\phi(K)\cong L\oplus [(K\oplus A)/\phi(K)]=L\oplus A',
		$$
we have contradicted that $L$ is relatively Hopfian.
		
Finally, the case where $K, L\in \mathcal {DF}$ is handled almost exactly as when $K, L\in \mathcal {RH}$, and so it is left to the reader.
\end{proof}
	
We are unable to decide whether or not the semi-rigid restriction on the pair of groups in this theorem can be removed.

\medskip
	
However, as two direct consequences, we derive immediately:
	
\begin{corollary} Suppose $G=R\oplus D$, where $R$ is a reduced group and $D$ is a divisible group. Then, $G$ is a member of any of the four classes $\mathcal H$, $\mathcal {RH}$, $\mathcal {WH}$ or $\mathcal {DF}$ if, and only if, both $R$ and $D$ are their members.
\end{corollary}
	
\begin{corollary}\label{torsionsplit} If $G=T\oplus L$ is a splitting mixed group, where $T$ is torsion and $L$ is torsion-free, then $G$ is in one of the four classes $\mathcal H$, $\mathcal {RH}$, $\mathcal {WH}$ or $\mathcal {DF}$ if, and only if, $T$ and $L$ are both members of the same class.
\end{corollary}

We now note the following pivotal variation on the definition of relative Hopficity.

\begin{proposition}\label{characterize} For a group $G$, the following are equivalent:
	
(a) $G$ is relatively Hopfian;
	
(b) For all primes $p$, one of the groups $G \oplus \Z(p)$, $G \oplus \qc$ is not a homomorphic image of $G$ or, equivalently, there is no surjective homomorphism $\gamma:G\to G\oplus Z$, where $p$ is a prime and either $Z\cong \Z(p)$ or $Z\cong \Z(p^\infty)$.
\end{proposition}

\begin{proof} Certainly (a) immediately insures (b). We prove the converse by contrapositive, so assume that $\phi:G\to G\oplus A$ is a surjection with $A\ne \{0\}$. If $A$ is not divisible, then, for some prime $p$, we have $\{\overline{0}\}\ne A/pA$. Consequently, there is clearly a surjective homomorphism $\xi: A\to \Z(p):=Z$. Suppose next that $A$ is divisible. If $A$ has non-zero $p$-torsion for some prime $p$, then there is a surjective homomorphism $\xi: A\to \Z(p^\infty):=Z$. And if $A$ is torsion-free, then, for all primes $p$, there is a surjective homomorphism $\xi: A\to \Z(p^\infty):=Z$. In either case, letting
    \[
	\gamma=(1_G\oplus \xi)\circ \phi:G\to G\oplus A\to G\oplus Z
	\]
shows that (b) also fails, finishing the proof.
\end{proof}

With this reformulation at hand, it is not too difficult to establish our next result. Nevertheless, we shall give a direct proof which seems to be more natural and transparent.

\begin{theorem}\label{suggestion} Suppose $G$ is a torsion-free group. Then, the following are equivalent:
	
(a) $G$ is Hopfian;
	
(b) $G$ is relatively Hopfian;
	
(c) $G$ is weakly Hopfian.
\end{theorem}

\begin{proof}
	One sees that (a) implies (b) implies (c) is true by Theorem~\ref{containments} listed above (notice that neither of these require that $G$ be torsion-free). To establish the result, it then suffices to show that (c) implies (a).
	So, suppose $G$ is weakly Hopfian (and torsion-free). If $\phi:G\to G$ is a surjective endomorphism with kernel $N$, then we need to show that $N=\{0\}$. But since $G$ is torsion-free and $G/N\cong G$, it follows at once that $N$ is pure in $G$. Thus, since $G$ is weakly Hopfian, we can conclude that $N=\{0\}$, and $G$ is then Hopfian.
\end{proof}

We now want to explore a simple but important question paralleling a known property of Hopfian groups: suppose that $A$ is a relatively Hopfian $p$-group and $B$ is a finite group, is the direct sum $A \oplus B$ still relatively Hopfian? The answer is {\it yes} if $A$ is Hopfian - see, for example, \cite{GG} or \cite{H1}. It follows then from Theorem\ref{suggestion} above that we may assume throughout that $A$ is {\it not} torsion-free. We focus now on the situation when $A$ is a $p$-group. Before proceeding to examine this situation, we mention first a result that considerably strengthens \cite[Proposition 113.3]{F1} in the restricted situation where we discuss $p$-groups.

\begin{lemma}\label{extending} Suppose $T$ is a $p$-group and $k<\omega$. If $\phi: p^k T\to p^k T$ is an epimorphism, then $\phi$ extends to an epimorphism $\gamma: T\to T$.
\end{lemma}

\begin{proof} By an easy induction, there is no loss of generality in assuming $k=1$.

Suppose $T=A\oplus B$, where $B$ is a maximal $p$-bounded summand of $T$. Let $N=p A\oplus B$ and $\mu:N \to N$ be $\phi$ on $p A$ and the identity on $B$; so, $\mu(N)=N$. Clearly, $\mu$ does not decrease heights computed in $T$. Since $T/N\cong A/pA$ is an elementary $p$-group, it is necessarily totally projective. Therefore, $\mu$ must extend to a homomorphism $\gamma:T\to T$. Certainly, as $\gamma$ is an extension of $\phi$, the only question is whether it is an epimorphism.

To show this, let $x\in T$. Thus, $$px\in pT=\phi(pT)=\gamma(pT).$$ Let $z\in T$ satisfy $p\gamma(z)=\gamma (pz)=px$. It now routinely follows that $$x- \gamma (z)\in T[p]=A[p]\oplus B\leq N\leq \gamma(T).$$ Suppose that $x-\gamma(z)=\gamma(y)$. Consequently, $x=\gamma(y+z)\in \gamma(T)$, as required.
\end{proof}

We are thus prepared to prove the following assertion.

\begin{proposition}\label{dirsum} Suppose $A$ is a relatively Hopfian $p$-group and $B$ is a finite $p$-group. Then, $A\oplus B$ is also relatively Hopfian.
\end{proposition}	

\begin{proof} Thanks to Proposition~\ref{characterize} above, we need to show that there is no surjective homomorphism $\phi: A\oplus B\to A\oplus B\oplus Z$, where $Z$ is either isomorphic to $\Z(p)$ or to $\Z(p^\infty)$. In this light, we will assume such a map $\phi$ exists and derive a contradiction.

To this target, suppose $k\geq 1$ with $p^k B=\{0\}$ and $A=A'\oplus C$, where $C$ is a maximal $p^k$-bounded summand of $A$. Since $A$ is relatively Hopfian, the Ulm factor of $A$, $U_n(A)$ is finite for all $n<\omega$; so, in particular, $C$ is finite. Note that $A'$, as a direct summand of a relatively Hopfian group, also satisfies that property applying Proposition~\ref{prop3}. Replacing now $A$ by $A'$ and $B$ by $B\oplus C$, we may assume with no loss in generality that $A[p]=(p^k A)[p]$.

\medskip

\noindent{\bf Case 1:} Suppose $Z=\Z(p)$: it then follows that $\phi$ restricts to a surjection
\[
p^k A=p^k(A\oplus B) \to p^k (A\oplus B\oplus Z)= p^k A.
\]
According to Lemma~\ref{extending}, this extends to an epimorphism $\gamma:A\to A$. Therefore, referring to Theorem~\ref{bigresult} listed below, the kernel of $\gamma$ is a bounded subgroup of $p^\omega A\subseteq p^k A$, so that both $\gamma$ and $\phi\restriction_{p^k A}$ have the same kernel, say $K$.

Note that $\gamma$ induces an isomorphism $A/K\cong A$. Since $\phi(K)=\{0\}$, there is an induced surjection
\[
\mu: A\oplus B \cong A/K\oplus B\to A\oplus B\oplus Z,
\]
which is the identity on $p^k A$. Since $p^k A[p]$ is essential in $A$, it follows that $\mu$ is an injection on $A$, i.e., $$A\cong \mu(A)=:A''\leq A\oplus B\oplus Z.$$
Since $$A''[p]=\mu(A[p])=\mu(p^k A[p])= p^k A[p]\leq A\oplus B\oplus Z,$$ we can conclude that in $A\oplus B\oplus Z$ we have $A''\cap (B\oplus Z)=\{0\}$. Consequently, $B\oplus Z$ maps injectively into $(A\oplus B\oplus Z)/A''$.
Thus, this quotient has order at least as large as the order of $B\oplus Z$.

However, $\mu$ induces a surjection
\[
B\cong (A\oplus B)/A\to \mu(A\oplus B)/\mu(A)= (A\oplus B\oplus Z)/A'',
\]
and since $B$ has strictly smaller order than $B\oplus Z$, we have our contradiction.

\medskip

\noindent{\bf Case 2:} Suppose $Z=\Z(p^\infty)$: in this case, $\phi$ restricts to a surjection

\[
p^k A=p^k (A\oplus B)\to p^k (A\oplus B\oplus Z)=p^k A\oplus Z.
\]
Let $\gamma: p^k A\oplus Z\to p^k A\oplus Z$ be $\phi$ on the first direct summand and 0 on second; so, $\gamma$ is obviously a surjection $p^k (A\oplus Z)\to p^k (A\oplus Z)$. Owing to Lemma~\ref{extending}, this extends to a surjection $\kappa:A\oplus Z\to A\oplus Z$. But since $\kappa(Z)=\{0\}$, it restricts to a surjection $\kappa\restriction_A : A\to A\oplus Z$, contradicting the assumption that $A$ is relatively Hopfian, as expected.
\end{proof}

As an important consequence, we may deduce the following criterion.

\begin{corollary}
Suppose $A$ is a $p$-group and $B$ is a direct sum of cyclic $p$-groups. Then, $A\oplus B$ is relatively Hopfian if, and only if, $A$ is relatively Hopfian and $B$ is finite.
\end{corollary}

\begin{proof} The \lq\lq if\rq\rq  \ part follows at once from Proposition~\ref{dirsum}.

Concerning the\lq\lq only if\rq\rq \ part, it follows from a direct combination of Propositions~\ref{prop3} and \ref{sums}.
\end{proof}
	
The following result is, however, rather unsurprising.
	
\begin{proposition}\label{allrel} Suppose $G$ is a torsion group with primary decomposition $G = \bigoplus\limits_{p \in \mathcal P}T_p$. Then, $G$ is in one of our four classes if, and only if, each primary component $T_p$ is in the corresponding class.
\end{proposition}
	
\begin{proof} We prove this for $\mathcal {RH}$ and leave the other three statements to the reader for inspection. If $G$ is relatively Hopfian then, since each primary component $T_p$ is a direct summand of $G$, it follows from Proposition \ref{prop3} that each $T_p$ is relatively Hopfian.
		
Conversely, suppose that each $T_p$ is relatively Hopfian, but $G$ is {\it not} relatively Hopfian. Thus, thanks to Proposition~\ref{characterize}, there is a prime $p$ and a surjective homomorphism $\gamma:G\to G\oplus Z$, where $Z$ is a non-zero $p$-group. It now follows that $\gamma(T_p)=T_p\oplus Z$, so that $T_p$ is not relatively Hopfian, as needed.
\end{proof}
	
We now include for the convenience of the reader and completeness of exposition a proof of the following well-known and helpful fact.
	
\begin{proposition}\label{size}
If $T$ is a $p$-group of finite $p$-rank and $B$ is a subgroup of $T$, then $C:=T/B$ also has finite $p$-rank, which satisfies the inequality $r_p(C)\leq r_p(T)$.
\end{proposition}
	
\begin{proof} Suppose $r_p(C)> n:=r_p(T)$. Let $\pi:T\to C$ be the canonical epimorphism. We can plainly find a finite subgroup $T'\leq T$ such that $C':=\pi (T')$ has $p$-rank $n+1$. It then follows that $\pi$ induces a surjection $T'/pT'\to C'/pC'$.
		Since $T'$ and $C'$ are direct sums of cyclic groups, it quickly follows that
		\[
		n\geq r_p(T')=r_p(T'/pT')\geq r_p(C'/pC')=r_p(C')=n+1,
		\]
giving our desired contradiction.
\end{proof}
	
The following assertion parallels the well-known (and easily checked) fact that {\it a torsion-free group of finite rank is Hopfian}.
	
\begin{corollary}\label{finiteprank}
A $p$-group of finite ($p$-)rank is relatively Hopfian.
\end{corollary}
	
\begin{proof}
Suppose $T$ is a $p$-group with $r_p(T)$ finite. If $T$ failed to be relatively Hopfian, owing to Proposition~\ref{characterize}, there would be a surjective homomorphism $\gamma:T\to T\oplus Z$, where $Z$ is a $p$-group of $p$-rank $1$. This would violate Proposition~\ref{size}, since $r_p(T\oplus Z)=r_p (T)+1$.
\end{proof}
	
What the last result says is that \textit{finitely co-generated} $p$-groups are always relatively Hopfian. We now completely characterize the divisible groups in our four classes.
	
\begin{proposition}\label{twosix} Suppose $D=D_0\oplus (\bigoplus_{p\in \mathcal P}D_p)$ is a divisible group with $D_0$ torsion-free and each $D_p$ a $p$-group.
		
(a) $D\in \mathcal {H}$ if, and only if, it is torsion-free of finite rank.
		
(b) $D\in \mathcal {RH}$, $\mathcal {WH}$ or $\mathcal {DF}$ if, and only if, $D_0$ has finite rank and, for every $p\in P$, $D_p$ has finite $p$-rank.
\end{proposition}
	
\begin{proof} Consulting with Corollary~\ref{cor3}, if $D$ is in any of these classes, then all of these ranks are finite.
		
For part (a), it is easy to show that any torsion-free group of finite rank is Hopfian. And since, for any prime $p$, the group $\Z(p^\infty)$ is {\it not} Hopfian, all of the groups $D_p$ must equal $\{0\}$.
		
For part (b), suppose all these ranks are finite. By (a), the group $D_0$ is Hopfian, and hence relatively Hopfian. Applying Corollary~\ref{finiteprank}, each $D_p$ is relatively Hopfian, so Proposition~\ref{allrel} gives that $\oplus_{p\in \mathcal P} D_p$ is also relatively Hopfian. So, employing Corollary~\ref{torsionsplit}, the group $D$ is relatively Hopfian, and hence simultaneously weakly Hopfian and directly finite, as required.
\end{proof}
	
\section{Torsion Groups}\label{tsect}
	
We begin our discussion here with an important lemma, which somewhat refines the technique used in \cite{Ke}.
	
\medskip
	
Concretely, the following is true:
	
\begin{lemma}\label{mequalzero} Suppose $G$ is a $p$-semi-standard $p$-group and $H$ is a subgroup of $G$ with $\min \val H = m$, a finite integer. If $\bar{G} = G/H \cong G$, then $m =0$.
\end{lemma}
	
\begin{proof} Since $\min \val H = m$, we must have $H \leq p^mG$. We claim $m= 0$: suppose, for a contradiction, that $m > 0$. Let $y \in H$ be chosen so that $\val y_G = m$ and $x \in G$ satisfies $px = y$ with $\val x = m-1$. Let $\{z_1, \dots , z_j\} \subseteq (p^{m-1}G)[p]$ be a maximal linearly independent set modulo $(p^mG)[p]$, so that $f_{m-1}(G) = j$, a finite integer by the hypothesis that $G$ is semi-standard. Letting $\bar{z_1} = z_1 +H, \dots , \bar{z_j} = z_j + H$, then these elements and the element $\bar{x} = x + H$ are all contained in $(p^{m-1}\bar{G})[p]$. To obtain the desired contradiction, it will suffice to show that the elements $\bar{x}, \bar{z_1}, \dots \bar{z_j}$ are linearly independent modulo $p^m\bar G$, since this would show that $$f_{m-1}(\bar{G}) \geq j+1 > j= f_{m-1}(G)$$ which is manifestly wrong since, by hypothesis, $\bar{G} \cong G$.
		
Suppose then that $$n_1 \bar{z_1} + \dots + n_j \bar{z_j} + k\bar{x} \in p^m\bar{G} = p^mG/ H,$$ because $H \leq p^mG$. We need to show that $p$ divides each of the coefficients $n_1, \dots , n_j, k$. Now, $$n_1 \bar{z_1} + \dots + n_j \bar{z_j} + k\bar{x} = \bar{w}$$ for some $w \in p^mG$. So, for some $h\in H$, we have  $$n_1z_1 + \dots + n_jz_j + kx = w+h \in p^mG+H=p^m G.  \quad \quad (1)$$ Replacing $w$ by $w+h$, we may assume $h=0$.
		
Multiplying by $p$, we get $$pkx = ky = pw \in p^{m+1}G.$$ If $p \nmid k$, then $\val y = \val {ky} \geq m+1$ - contrary to the definition of $y$. So, $k = pk'$ for some integer $k'$. Substituting back into (1), we have $$ n_1z_1 + \dots + n_jz_j = w - pk'x = w - k'y \in p^mG$$ and the linear independence of the $z_i$ modulo $p^mG[p]$ now gives $p \mid n_i$ for $i = 1, \dots j$, as required.   	
\end{proof}
	
We are now in a position to prove a principal result clarifying the relationship between Hopfian and relatively Hopfian $p$-groups.
	
\begin{theorem}\label{bigresult}
Suppose $G$ is a $p$-group. Then, $G$ is relatively Hopfian if, and only if, whenever $H$ is a subgroup of $G$ such that $G/H\cong G$, then $H$ is a bounded subgroup of $p^\o G$.
\end{theorem}
	
\begin{proof} Considering sufficiency, suppose the later condition holds for all such subgroups $H$. If $G$ is not relatively Hopfian, then Proposition~\ref{characterize} enables us that there is a surjective homomorphism $\gamma: G\to G\oplus Z$, where $Z\cong \Z(p)$ or $Z\cong \Z(p^\infty)$. If $H=\gamma^{-1}(Z)$, then it follows that $G/H\cong G$, so that $H$ is a bounded subgroup of $p^\omega G$. Note that $\gamma$ restricts to a surjective homomorphism $H\to Z$, and since $H$ is bounded, we must have $Z\cong \Z(p)$. And since $H\subseteq p^\o G$, we must be that $$\gamma(H)\subseteq p^\o (G\oplus Z)=p^\o G\oplus \{0\},$$ which contradicts that $\gamma(H)=Z$.
		
Conversely, suppose $G$ is relatively Hopfian and $G/H \cong G$. If $H$ is unbounded, then it has a proper basic subgroup $B$ so that $H/B \neq \{\overline 0\}$ is divisible. But then $G/B = (H/B) \oplus (K/B)$ for some $K$ with $G \cong K/B$, contrary to $G$ being relatively Hopfian. So, $H$ is bounded.
		
Note also that, as $G$ is relatively Hopfian by assumption, it follows that $G$ is $p$-semi-standard. Now, if $H$ is contained in $p^{\omega}G$, we are finished. Otherwise, $m:=\min \val H<\omega$ and so applying Lemma \ref{mequalzero}, we must have $m = 0$. However, this would result, with the validity of Lemma \ref{hzero}, in $G$ not being relatively Hopfian. Thus, $H$ must be contained in $p^{\omega} G$, as required.
\end{proof}
	
We can now deduce the following interesting equivalence between the notions of Hopficity and relative Hopficity for certain classes of $p$-groups.
	
\begin{corollary}\label{ulmhopfian}
Suppose $G$ is a $p$-group such that $p^\omega G$ is Hopfian (in particular, if $p^\omega G$ is finite or $G$ is separable). Then, $G$ is relatively Hopfian if, and only if, it is Hopfian.
\end{corollary}
	
\begin{proof} We know that if $G$ is Hopfian, then it is relatively Hopfian. So, assume $G$ is relatively Hopfian. To show $G$ is Hopfian, suppose $H$ is a subgroup of $G$ such that $G/H\cong G$. Exploiting Theorem~\ref{bigresult}, we can conclude that $H\subseteq p^\omega G$. Therefore, $p^\o G/H=p^\o (G/H)\cong p^\o G$. But since $p^\o G$ is Hopfian, we can conclude that $H=\{0\}$, so that $G$ is Hopfian, as promised.
\end{proof}
	
Our next result allows us to weaken the requirement of Hopficity of the first Ulm subgroup in Corollary \ref{ulmhopfian} above.
	
\begin{proposition}\label{ulmbounded}
If $G$ is a $p$-group such that $p^\o G$ is bounded and $G/p^\o G$ is Hopfian, then $G$ is relatively Hopfian.
\end{proposition}
	
\begin{proof}
If, in a way of contradiction, $G$ is not relatively Hopfian, then with the help of Proposition~\ref{characterize} there is a surjective homomorphism $\gamma: G\to G\oplus Z$, where $Z\cong \Z(p)$ or $Z\cong Z(p^\infty)$. Evidently, $$\gamma(p^\o G)\subseteq p^\o (G\oplus Z)= p^\o G \oplus p^\o Z.$$
		
Suppose first that $Z\cong \Z(p^\infty)$. Since $p^\omega G$ is bounded and $Z$ is unbounded, it follows that $\gamma (p^\o G)\ne p^\o G\oplus Z$. And if $Z\cong \Z(p)$, then $p^\o Z=\{0\}$, so again $\gamma (p^\o G)\ne p^\o G\oplus Z$.
		
Therefore, in either case, there is a surjective composite-homomorphism
		\[
		G/p^\o G\to (G\oplus Z)/\gamma(p^\o G)\to (G\oplus Z)/(p^\o G\oplus Z)\cong G/p^\o G,
		\]
which is not injective. This, however, contradicts that $G/p^\o G$ is Hopfian, and completes the proof.
\end{proof}
	
We have seen in Corollary~\ref{ulmhopfian} above that if $G$ is a reduced $p$-group of length $\o$, then $G$ is relatively Hopfian if, and only if, it is Hopfian. An immediate question arises as to whether or not this result can be extended to groups having greater length; in particular, does it expand to groups of length exactly $\o +1$? We shall show by constructing a suitable example that there is a relatively Hopfian, but non-Hopfian group $H$ of length exactly $\o +1$ with $p^\o H$ an elementary $p$-group of uncountable rank.
	
Our construction is quite complex and is based on ideas of both Pierce and Corner; in particular, we shall utilize Corner's well-known realization theorem for separable $p$-groups. For presentation purposes, we have separated out a part of the argument as a preliminary technical lemma.
	
\begin{lemma}\label{nonHprelim}
Suppose $G$ is an unbounded semi-standard separable Hopfian $p$-group with a basic subgroup $B$ such that: (i) there is an automorphism $\a$ of $G$ whose restriction to $B$ is an automorphism of $B$, and (ii) there is an element $z \in G[p] \setminus B[p]$ with $f(\a)(z) \notin B$ for any polynomial $f$ in $\Z[\a] \setminus p\Z[\a]$.
	
Then, there exists a subgroup $K$ of countable rank of $G[p]$ with $\a(K) \leq K$, $z\not\in K$ but $\a(z) \in K$. Furthermore, the group $H = G/K$ has a first Ulm subgroup which is elementary of uncountable rank and $H$ is relatively Hopfian but {\it not} Hopfian.
\end{lemma}
	
\begin{proof} 		
Let $L = \langle \a^n(z) : n \geq 1\rangle$ so that it follows from our second hypothesis that $L \cap B[p] = \{0\}$; set $K = B[p] \oplus L$ and note that, as $B[p]$ is invariant under $\a$, $\a(K) \leq K$ and also $\a(z) \in L \leq K$; the countability of $K$ follows from the fact that both $B, L$ are countable, thereby establishing the first claim.
	
Now, consider $H = G/K$. Since $G$ is separable and $B[p] \leq K \leq G[p]$, a straightforward argument using standard properties of basic subgroups yields that $p^\o(G/K) \cong G[p]/K$, so that $p^\o H$ is certainly elementary.
	
Note also that, by hypothesis, $G$ is Hopfian and so $G$, and hence $G[p]$ is uncountable, which shows that $p^\o H$ is certainly of infinite, even uncountable, rank.
	
Further, one checks that $$H/p^\o H \cong G/G[p] \cong pG$$
and, as $G$ is Hopfian, so too is $pG$, as is well-known and easy to prove. Applying Proposition~\ref{ulmbounded} to $H$, we see that $H$ is certainly relatively Hopfian. Finally, as $\a: G \twoheadrightarrow G$ is a surjection with $\a(K) \leq K$, there is an induced surjection $\bar{\a}: H \twoheadrightarrow H$. But, $z + K$ is a non-zero element of $\Ker\ \bar{\a}$, so that $H \cong H/\Ker\ \bar{\a}$ is not Hopfian, as claimed.
\end{proof}

We now have all the machinery to produce our promised example.

\def\ii#1{{#1}^{-1}}
\def\ni{\noindent}	

\begin{example}\label{piercecorner} If $p$ is a prime, then there is a $p^{\o+1}$-bounded $p$-group $H$ with infinite $p^{\o}H$ such that $H$ is relatively Hopfian, but {\it not} Hopfian.
\end{example}
	
\begin{proof} Let $t$ be an indeterminate and $R$ be the $p$-adic completion of the ring $\Z[t,\ii t]$. We let $$\overline R=R/pR\cong \Z(p)[t,\ii t].$$
Likewise, let ${\bf p}_1(t), {\bf p}_2(t), \dots$ be a listing of monic polynomials in $\Z[t]$ of degree at least $1$ such that if $\overline {\bf p}_n(t)$ is the projection of ${\bf p}_n(t)$ into $\overline R$, then $\overline {\bf p}_1(t), \overline {\bf p}_2(t), \dots$ is a list of distinct monic irreducible polynomials in $\Z(p)[t]$ {\it other than $t$}.
		
\medskip
		
For each $n\in \N=\{1, 2, 3, \dots\}$, let $d_n$ be the degree of ${\bf p}_n(t)$, so that
		\[B_n:=\Z[t]/(p^n,{\bf p}_n(t))\cong \Z(p^n)^{d_n}\]
To show validity of the last isomorphism, one routinely checks that the surjective homomorphism $\Z\to \Z(p^n)$ with kernel $(p^n)\subseteq \Z$ determines a surjective homomorphism $\Z(t)\to \Z(p^n)[t]$ with kernel $(p^n)\subseteq \Z[t]$. Since $\overline {\bf p}_n(t)$ is monic, it divides into every other element of $\Z(p^n)[t]$ with a unique quotient and remainder of degree less than $d_n$. This yields that, as abelian groups,
		\[
		\Z[t]/(p^n,{\bf p}_n(t))\cong \Z(p^n)[t]/(\overline {\bf p}_n(t))\cong \bigoplus_{j=0}^{d_n-1} \langle t^j\rangle \cong \Z(p^n)^{d_n},
		\]
as pursued.
		
Note that $B_n$ is clearly a $\Z[t]$-module.  If $q(t)\in \Z[t]$ is such that $\overline q(t)$ is relatively prime to $\overline {\bf p}_n(t)$ in $\Z(p)[t]$, it follows that multiplication by $q(t)$ restricts to an automorphism of $B_n[p]\cong \Z(p)[t]/(\overline {\bf p}_n(t))$, so that it is, in fact, an automorphism on $B_n$. In particular, multiplication by $t$ induces an automorphism of $B:=\oplus_{n\in \N} B_n$, which we denote by $\tau$.
		
This implies that there is a natural $R$-module structure on $B$, and hence on its torsion completion $\overline B$. We next want to show that this construction satisfies Corner's ``Condition (C)'' (see \cite [Theorem~2.1]{C3}): suppose $\b\in R\setminus pR$; we need to show $\b$ is not zero on $(p^m B)[p]$ for any $m\in \N$. Note $\b = t^j q(t)$ for some $j\in \Z$, where $\overline q(t)\in \Z(p)[t]$. Since $t$ induces the isomorphism $\tau$ on $B$, we may assume that $j=0$ and $\b(t)=q(t)$. Choose $n>m$, so that $\overline {\bf p}_n(t)$ is relatively prime to $\b(t)$ when mapped into $\Z(p)[t]$ (i.e., the irreducible $\overline {\bf p}_n(t)$ does not divide $\overline \b(t)$).  Therefore, $\b(t)$ induces an automorphism of $B_n$, so that $\b((p^m B)[p])\ne \{0\}$.
		
It follows from the cited Corner's construction that there is a separable $p$-group $G$ with basic subgroup $B$ such that $E(G)=R\oplus S$, where $S$ is the ideal of small endomorphisms of $G$, and $R$ restricts to the above endomorphisms on $B$. If $\b(t)\in R$ and $x\in G$, we will often denote $\b(t) x$ by $\b(x)$.
		
Observe elementarily that, if $m\in \N$, then $G=B_{\leq m}\oplus C_{m}$, where $B_{\leq m}=\oplus_{n\leq m} B_n$ and $C_{m}$ is the closure of $B_{>m}:=\oplus_{n>m} B_n$ in $G$. If $\b\in R$, then clearly $\b(B_{\leq m})\subseteq B_{\leq m}$. And since $\b(B_{>m})\subseteq B_{>m}$, it easily follows that $\b(C_m)\subseteq C_m$.
		
If $0\ne \overline \b(t)\in \overline R$, and $m$ is chosen so that, for each $n>m$, $\overline {\bf p}_n(t)$ is relatively prime to $\overline \b(t)$ in $\overline R$, then it follows that $\b(t)$ restricts to an automorphism of $B_{>m}$. This readily ensures that $\b(t)$ preserves the $p$-heights of all members of $C_m[p]$. Therefore, $\b(t)$ embeds $C_m$ as a pure subgroup of $C_m$, and since it contains $B_{>m}$, we can conclude that $C_m/\b(C_m)$ is divisible.
		
\medskip

Pausing for some standard notation, for $x\in G$, we let $p^{e(x)}$ be the order of $x$.
		
\medskip

\ni {\bf Claim 1:} Suppose $\lambda=p^k \b\in R$, where $\b\in R$ has $p$-height 0 (so $\lambda$ has $p$-height $k\in \o=\{0,1,2, \dots\}$). Then, there is an $m\in \N$ such that if $x\in G$ has $e(x)>m$, then $e(\lambda(x))=e(x)-k.$
		
\medskip
		
Choose $m\in \N$ such that $m>k$ and no $\overline {\bf p}_n(t)$ for $n> m$ divides $\overline \b$ in $\overline R$. Consider the decomposition $G=B_{\leq m}\oplus C_m$ as above. If $e(x)>m$, it follows that $x=x_1+x_2$, where $x_1\in B_{\leq m}$ and $x_2\in C_m$, and $e(x_1)\leq m<e(x)=e(x_2)$. Again $\b$ restricts to a pure embedding $C_m\to C_m$ and another homomorphism $B_{\leq m}\to B_{\leq m}$. It, thus, follows that
		\[
		e(x)-k\geq    e(\lambda(x))=\max\{e(p^k \b(x_1)), e(p^k \b(x_2))\}\geq e(p^k \b(x_2))=e(x)-k,
		\]
as claimed.
		
\medskip
		
\ni {\bf Claim 2:} Suppose $\gamma = \lambda +\phi$, where $\lambda=p^k \b$ as above, and $\phi\in S$ is a small endomorphism of $G$. Then, there is an $m\in \N$ such that
		\[
		\gamma(G)\cap  G[p^{m+1}]\subseteq p^k \b(G)+G[p^m].  \eqno {(1)}
		\]
		
\medskip
		
Choose $m>k$ such that, if $e(x)>m$, then
		
\medskip
		
\ni (a) $e(\lambda(x))=e(x)-k$ (This $m$ exists by Claim 1);
		
and
		
\ni (b) $e(\phi(x))\leq e(x)-(k+1)=e(\lambda(x))-1$. (This $m$ exists by \cite[Theorem~3.3]{P}).
		
\medskip
		
Suppose

        \[y=\gamma(x)=p^k \b(x) +\phi(x)\in G[p^{m+1}].
		\]  It follows now that $p^{m+1} y=0$, or
		\[
		p^{m+1} \lambda(x)=-p^{m+1} \phi(x). \eqno{(2)}
		\]
We assert that $p^m \phi(x)= 0$, which will establish the present claim. Assume otherwise, so that $p^m \phi(x)\ne 0$ and hence $p^m x\ne 0$. Therefore, $e(x)>m$, so that using (b) above,
		\[
		m<e(\phi(x))\leq e(x)-(k+1).
		\]
Thus, utilizing (a) above, $$m+1<e(x)-k=e(\lambda(x)),$$ so the terms in (2) are non-zero. Consequently, $e(\phi(x))=e(\lambda(x))$. This, however, contradicts (b), establishing Claim 2.
		
\medskip
		
\ni {\bf Claim 3:} $G$ is Hopfian.
		
\medskip
		
Continuing with the above notation, suppose $\gamma$ is surjective; we want to show it is an isomorphism. We first observe that $k=0$, so that $\lambda=\b$: if $k>0$, then using the number $m$ from Claim 2, if $y\in B_{m+1}$ has height $0$, then $y\in G[p^{m+1}]$ is not in $p^k \b(G)+G[p^m]\subseteq p^k G+G[p^m]$, so $y$ cannot be in $\gamma(G)$ and $\gamma$ is not surjective.
		
Thereby, we may assume $\gamma=\b+\phi$, so that
		\[
		G[p^{m+1}]\subseteq \b(G)+G[p^m].  \eqno{(3)}
		\]
		
Again, in the decomposition $G=B_{\leq m}\oplus C_m$, since $\b(B_{\leq m})\subseteq B_{\leq m}$ and $\b(C_m)\subseteq C_m$, (3) implies that
		\[
		C_m[p^{m+1}]\subseteq \b(C_m)+C_m[p^m].
		\]
Or, in other words,
		\[
		\b(C_m)+C_m[p^{m+1}]=\b(C_m)+C_m[p^m].
		\]
		
We already noted before that $\b(C_m)$ will be pure and dense in $C_m$. So, we subsequently deduce
		\begin{align*}
			(C_m/\b(C_m))[p^m] &= (\b(C_m)+C_m[p^m])/\b(C_m)\cr
			&= (\b(C_m)+C_m[p^{m+1}])/\b(C_m)\cr
			&=(C_m/\b(C_m))[p^{m+1}].
		\end{align*}
Therefore, $C_m/\b(C_m)=\{\overline{0}\}$, i.e., $\b(C_m)=C_m$, so that $\b$ restricts to an automorphism on $C_m$.
		
We now observe that $\phi((p^m G)[p])=\{0\}$: indeed, assume $0\ne z\in (p^m G)[p]$. If $z=p^m x$, then $e(x)=m+1>m$. So, by (b), one verifies that $e(\phi(x))\leq e(x)-1=m$. Thus, $$\phi(z)=p^m \phi(x)=0,$$ as stated.
		
Next, if $\pi:G\to C_m$ is the obvious projection, then since $\phi(C_m[p])=\{0\}$, the map $\b$ agrees with $\pi\circ \gamma$ on $C_m[p]$. Furthermore, we also have $(\pi\circ \gamma)(C_m)=C_m$, i.e., the composition $\pi\circ \gamma$ is also an automorphism of $C_m$. This yields that $$G=B_{\leq m}\oplus \gamma(C_m).$$ Since $\gamma$ induces a surjective homomorphism $$B_{\leq m}\cong G/C_m\to G/\gamma(C_m)\cong B_{\leq m}$$ and as $B_{\leq m}$ is finite, the map $\gamma$ induces an isomorphism $G/C_m\to G/\gamma(C_m)$. And since $\gamma$ too induces an isomorphism $C_m\to \gamma(C_m)$, it follows at once that $\gamma$ itself is an isomorphism, verifying after all Claim 3.

Let $z\in G[p]\setminus B[p]$. We assert now that, if $$\s(t)\in \Z[t]\setminus p\Z[t],$$ then $\s(z)\not\in B$: in fact, expressing $z=\sum_{n\in \N} z_n$, where $z_n\in B_n[p]$, then it must be that $z_n\ne 0$ for an infinite number of such $n$. For all but finitely many $n$, $\s(t)$ is an isomorphism on $B_n$, so that $\s(z)\not\in B$.

It follows that $$L:=\langle \tau^n(z): n\in \N\rangle\subseteq G[p]$$ satisfies the relations $$L\cong (t)\subseteq \Z(p)[t] ~ {\rm and} ~ L\cap B[p]=\{0\}.$$
The proof now follows from an appeal to Lemma \ref{nonHprelim}, with $\tau$ playing the role of $\a$ in this lemma.		 \end{proof}
	
We now completely describe the weakly Hopfian $p$-groups discovering their structure. Before doing that, we prove the following non-trivial technicality, which must surely be known to experts in the field, but seems never to have appeared in the standard literature.

\begin{lemma}\label{good}
If $N$ is a pure subgroup of $G$ and $k<\omega$, then the Ulm functions satisfy the equality
$$f_k(G) = f_k(G/N) + f_k(N).$$
\end{lemma}

\begin{proof} For any group $H$, let $U_k(H)$ denote  the Ulm factor $(p^k H)[p]/(p^{k+1} H)[p]$. Consider the following diagram:
	
	\[\begin{CD}
		@. \{0\}   @.  \{0\}    @. \{0\}    \\
		@.   @VVV               @VVV             @VVV  \\
		\{0\}   @>>>  p^{k+1} N[p]    @>>>  p^{k+1} G[p]    @>>> p^{k+1}(G/N)[p] @>>> \{0\}   \\
		@.   @VVV               @VVV             @VVV  \\
		\{0\}   @>>>  p^{k} N[p]    @>>>  p^{k} G[p]    @>>> p^{k}(G/N)[p] @>>> \{0\}   \\
		@.   @VVV               @VVV             @VVV  \\
		\{0\}   @>>>  U_k(N)    @>>>  U_k(G)    @>>>    U_k(G/N) @>>> \{0\}   \\
		@.   @VVV               @VVV             @VVV  \\
		@. \{0\}   @.  \{0\}    @. \{0\}    \\
	\end{CD}\]
	
\medskip
\medskip
	
One plainly inspects that its upper two horizontal rows are short exact by purity and the fact that $k<\omega$. Likewise, its columns are short exact by definition. So, it follows that the bottom row is also short exact. Since these are all $p$-bounded groups, the bottom row actually splits. Considering their ranks, we get the wanted equality.
\end{proof}

It is worth noticing that a more direct argument, based on the well-known fact that neatness of a subgroup $H$ of $G$ is equivalent to the equality $(G/H)[p] = (G[p] +H))/H$, can be given. However, our homological proof above is more transparent and conceptual. In this in mind, it is also possible to expand the result to subgroups which are both nice and isotype, i.e., to balanced subgroups.

\medskip

We can now give a satisfactory description of weak Hopficity for primary groups.

\begin{theorem}\label{torsionweak} If $G$ is a $p$-group, then the following three items are equivalent:
	
(a) $G$ is weakly Hopfian;
	
(b) $G$ does not have a summand of the form $A^{(\kappa)}$, where $A$ is co-cyclic (i..e,  $A\cong \Z(p^k)$ for some $k\in \N^\infty:=\N\cup \{\infty\}$). In other words, $f_k(G)$ is finite for all $k\in \o^\infty=\o\cup \{\infty\}$;
	
(c) $G$ is directly finite.
\end{theorem}

\begin{proof} $\lnot$(a) $\Rightarrow  \lnot$(b): Suppose $G$ is not weakly Hopfian. Then, it has a non-zero pure subgroup $N$ such that there is an isomorphism $G/N\cong G$. Suppose first that $N$ is divisible, and denote the maximal divisible subgroup of $G$ by $D$. It follows that
	\[
	G\cong (G/N)\oplus N\cong G\oplus N,
	\]
so that $D\cong D\oplus N$. This implies that $D$ has infinite rank, so that $G$ has a summand of the form $\Z(p^\infty)^{(\o)}$, as required.
	
If $N$ fails to be divisible, then for some $k<\o$ we have $f_k(N)\ne 0$. The purity of $N$ is a guarantor that, in view of Lemma~\ref{good}, our Ulm-Kaplansky functions will satisfy the following two equalities
	\[
	f_k(G)=f_k(G/N)+f_k(N)=f_k(G)+f_k(N).
	\]
This forces that $f_k(G)$ is infinite, which means that $G$ has a summand of the form $\Z(p^{k+1})^{(\o)}$, which proves the implication.
	
$\lnot$(b) $\Rightarrow  \lnot$(c): We indicated this previously.
	
$\lnot$(c) $\Rightarrow  \lnot$(a): Follows from $\mathcal {WH}\subseteq \mathcal{DF}$.
\end{proof}

As a consequence, we extract the following necessary and sufficient condition.

\begin{corollary}\label{cor7}
A reduced $p$-group is weakly Hopfian if, and only if, it is $p$-semi-standard.
\end{corollary}

There are clearly lots of examples of weakly Hopfian $p$-groups that fail to be relatively Hopfian; e.g., any unbounded $p$-semi-standard direct sum of cyclic $p$-groups will suffice.

\medskip

It follows from Corollary \ref{cor7} that a reduced relatively Hopfian $p$-group has cardinality at most the continuum. In the case where the cardinality is actually countable, we obtain:

\begin{proposition}\label{count} A countable reduced relatively Hopfian $p$-group is finite.
\end{proposition}

\begin{proof} As already observed, a reduced relatively Hopfian $p$-group is semi-standard, hence if $G$ is bounded it must be finite. Suppose then that $G$ is unbounded and countably infinite. Then, a well-known consequence of Zippin's theorem - see, for instance, \cite[Proposition 1.11, Chapter 11]{F} - is that $G$ has a direct summand, say $H$, which is an unbounded direct sum of cyclic groups. Therefore, thanks to Proposition~\ref{prop3}, we discover that $H$ is simultaneously relatively Hopfian and an unbounded direct sum of cyclic groups; in particular, $H$ is separable and so with the aid of Corollary \ref{ulmhopfian} we get that $H$ is Hopfian. This is, however, impossible since, as is well known and easy to prove, an unbounded direct sum of cyclic groups is never Hopfian. Consequently, $G$ is necessarily finite, as formulated.
\end{proof}

We now deduce the useful consequence:

\begin{corollary} A reduced relatively Hopfian $p$-group cannot have a countably infinite direct summand. In particular, a reduced unbounded totally projective $p$-group is never relatively Hopfian.
\end{corollary}

\begin{proof} Since in view of Proposition~\ref{prop3} a summand of a relatively Hopfian $p$-group is again relatively Hopfian, the first statement follows immediately from Proposition \ref{count}. Furthermore, a reduced unbounded totally projective $p$-group always admits a direct summand which is an infinite direct sum of cyclic groups - see, for instance, \cite[Theorem 4.5(1)]{DGSZ} - and hence cannot be relatively Hopfian, as claimed.
\end{proof}  	

Apart from the totally projective $p$-groups, the other significant class of $p$-groups with a good classification is the class of reduced torsion-complete groups (cf. \cite{F1},\cite{F}). In what follows, it is easy to show, with what we have established so far in hand, that an infinite torsion-complete $p$-group is never relatively Hopfian. Before proving that, we are searching for the next technicality (for a relevant result, see \cite[Theorem 3.3]{Ka} as well).

\medskip

For completeness of the exposition, we recall that a subgroup $H$ of a group $G$ is said to be {\it dense}, provided the factor-group $G/H$ is divisible.

\begin{lemma}\label{extend} Let $G$ be a reduced torsion-complete $p$-group and $H\leq G$ a pure dense subgroup. Then, every surjective endomorphism $f: H\to H$ can be extended to an unique surjective endomorphism $\varphi: G\to G$.
\end{lemma}

\begin{proof} Since $G$ is torsion-complete, one knows that $f$ is extendible to an unique endomorphism $\varphi$ of $G$ (see \cite[Volume II]{F1}). We now make sure that $\varphi$ is, in fact, surjective. To that goal, suppose $g\in G$ and $g=\lim_i h_i$ with $h_i\in H$, where $h_{i+1}-h_i=p^ix_i$ with $x_i\in H$, i.e., $\{h_i\}_i$ is a neat convergent Cauchy sequence. Letting $\varphi(y_1)=x_1$, $y_1\in H$ and $c_2:=c_1+py_1$, where $\varphi(c_1)=h_1$ and $c_1\in H$, one discovery that $\varphi(c_2)=h_1+px_1=h_2$.
	
Assume now that there exist elements $c_1,\dots,c_n\in H$ such that $\varphi(c_i)=h_i$ for $i=1,\dots, n$ and $c_{k}-c_i\in p^iH$ for $i\leq k\leq n$. Since $$h_{n+1}-h_n=p^nx_n, x_n\in H,$$ one detects that $\varphi(y_n)=x_n$ for some $y_n\in H$. Moreover, if $c_{n+1}:=c_n+p^ny_n$, then $\varphi(c_{n+1})=h_{n+1}$ and $c_{n+1}-c_n\in p^nH$. Furthermore, an easy check gives that $c_{k}-c_n\in p^nH$ for all $k\geq n$, i.e., $\{c_n\}_n$ is a pure sequence of Cauchy.
	
If now $c=\lim_n c_n$, then it follows at once that $\varphi(c)=g$. Consequently, $\varphi$ is surjective, as pursued.
\end{proof}

Thus, we now have at our disposal all the necessary information to prove the following.

\begin{proposition}\label{prop216} If $\bar{B}$ is an infinite torsion-complete $p$-group with a basic subgroup $B$, then $\bar{B}$ is not relatively Hopfian.
\end{proposition}

\begin{proof} If $\bar{B}$ is bounded, then the result follows as in the proof of Proposition \ref{count}.
	
So, suppose that $\bar{B}$, and hence $B$ is unbounded. So, $B$ has a direct summand, $C$ say, of the form $$C = \bigoplus\limits_{i \geq 1} \langle e_i\rangle, ~ \text{where} ~ o(e) = p^{n_i} ~ \text{and} ~ n_1 < n_2 < \dots.$$ Now, if $\bar{B}$ is relatively Hopfian, then it follows from Corollary \ref{ulmhopfian} that $\bar{B}$ is Hopfian and so its direct summand $\bar{C}$ is also Hopfian according to Proposition~\ref{prop3}. However, the left Bernoulli shift $\sigma:C \to C$ defined as $$\sigma(e_1) = 0, \sigma(e_{i+1}) = e_i$$ for all $i \geq 1$ gives a surjective endomorphism of $C$ with non-trivial kernel. Now, $\sigma$ obviously extends to an endomorphism $\bar{\sigma}$ of $\bar{C}$ which again is surjective. In fact, a quick look at Lemma~\ref{extend} shows that it is workable in our situation, and thus it is clear now that $\overline \s$ is really onto and has non-trivial kernel. However, as the isomorphism $\bar{C}/ \Ker \ \bar{\sigma} \cong \bar{C}$ is valid, and since $\Ker \ \bar{\sigma} \neq \{0\}$, we thus have obtained a contradiction to $\bar{C}$ being Hopfian. So, no infinite torsion-complete $p$-group is relatively Hopfian, as expected.
\end{proof}

\section{Mixed Groups}\label{mixedgroups}

Throughout this section, recall that $T(G)=t(G)$ will denote the torsion part (often called the (maximal) torsion subgroup) of a group $G$ with $p$-components of torsion $T_p(G)=t_p(G)$. Likewise, hereafter, $G^1=\bigcap_{n\geq 1} nG$ stands for the {\it first Ulm} subgroup of $G$ (see cf. \cite{F1,F}).

\medskip

Though we are primarily concerned here with mixed groups, we state the next result in somewhat greater generality.

\begin{theorem}\label{lemma221} Suppose $B$ is a fully invariant subgroup of $G$ and $G/B$ is Hopfian. If $B$ is in one of $\mathcal H$, $\mathcal {RH}$, $\mathcal {WH}$ or $\mathcal {DF}$, then $G$ is in the same class.
\end{theorem}

\begin{proof} Suppose $B$ is Hopfian, and $\pi:G\to G$ is a surjection with kernel $J$; we need to show $J=\{0\}$. The full invariance of $B$ implies that $\pi(B)\subseteq B$, so that there is a surjective homomorphism
	$ G/B\to G/\pi(B)\to G/B$. Since $G/B$ is Hopfian, we can infer that this homomorphism is injective. Therefore, $J\subseteq B$ and $B=\pi(B)$. So, since $B$ is Hopfian, we can conclude that $J=\{0\}$, as desired.
	
The proof when $B$ is weakly Hopfian is almost exactly the same, only supposing $J$ is pure in $G$; it is, therefore, left to the reader.
	
Suppose $B$ is relatively Hopfian. Given a surjective homomorphism $\phi:G\to G\oplus A$, we want to show $A=\{0\}$. The full invariance of $B$ shows that $\phi(B)\subseteq B\oplus A$.
	
It follows that the composition
	\[
	G/B\to (G\oplus A)/\phi(B)\to (G\oplus A)/(B\oplus A)\cong G/B
	\]
will be surjective, and hence injective. Thus, $\phi(B)=B\oplus A$. And since $B$ is relatively Hopfian, we can derive $A=\{0\}$, as desired.
	
The proof when $B$ is directly finite is almost exactly the same, except $\phi$ will be an isomorphism; so it is also left to the reader.
\end{proof}

The following result, which is a direct consequence of Theorem~\ref{lemma221}, shows that one direction in the analogue of Corollary~\ref{torsionsplit} for \textit{non-torsion-splitting} mixed groups is, actually, valid.

\begin{corollary}\label{quotients} Suppose $G$ is group with torsion subgroup $T$ such that $G/T$ is Hopfian (equivalently, relatively Hopfian or weakly Hopfian). If $T$ is in $\mathcal H$, $\mathcal {RH}$, $\mathcal {WH}$ or $\mathcal {DF}$, then $G$ is in the same class.
\end{corollary}

And since every finite rank torsion-free group is Hopfian, and hence relatively Hopfian, we have following special case of the last result:

\begin{corollary}\label{lemma220} If $G$ is a group of finite torsion-free rank and its torsion subgroup is in $\mathcal H$, $\mathcal {RH}$, $\mathcal {WH}$ or $\mathcal {DF}$, then $G$ is in the same class.
\end{corollary}

We also have the following partial converse to this result:

\begin{corollary}
If $G$ is a weakly Hopfian mixed group, then its torsion subgroup $T$ is weakly Hopfian.
\end{corollary}

\begin{proof} If we assume the opposite that $T$ is not weakly Hopfian, then thanks to Proposition~\ref{allrel} and Theorem~\ref{torsionweak} there is a prime $p$ and a summand of $T$ of the form $A^{(\o)}$, where $A$ is a co-cyclic $p$-group. This summand of $T$ will be either bounded or divisible. Since all such groups are pure-injective, this is also a direct summand of $G$, so that $G$ is not weakly Hopfian, contradicting the initial hypothesis.
\end{proof}

The next result shows the equivalence of relative Hopficity and ordinary Hopficity for arbitrary direct sum of cyclic groups.

\begin{proposition}\label{sums} If $G$ is a direct sum of cyclic groups, then the following are equivalent:

\medskip
	
(i) $G$ is relatively Hopfian;

\medskip

(ii) $G$ is Hopfian;

\medskip

(iii) The torsion-free part and each $p$-component have finite ranks.
\end{proposition}


\begin{proof} Let $G = F \oplus T$, where $F$ is free and $T = \bigoplus_pT_p$ is the primary decomposition of $T$. Then (ii) implies (iii) by standard properties of Hopfian groups; (iii) implies (i) follows from Theorem~\ref{lemma221} since $T$ is fully invariant in $G$. Thus, it remains to show that (i) implies (ii).
	
So, suppose that $G = F \oplus T$ is relatively Hopfian. Then, $F$ and each primary component $T_p$ are relatively Hopfian as direct summands of $G$. Now, $F$ is also torsion-free, so appealing to Theorem~\ref{suggestion} the group $F$ is Hopfian. Furthermore, each $T_p$ is separable, so it follows from Corollary~\ref{ulmhopfian} that each $T_p$ is Hopfian. An application of Proposition~\ref{allrel} gives that $T$ is then Hopfian and, as $T$ is fully invariant in $G$ with torsion-free quotient, the group $G$ itself is Hopfian by a well-known property of Hopfian groups, as pursued.
\end{proof}

We continue to explore situations in which relatively Hopfian groups are always Hopfian.

\begin{proposition}\label{cotorsion} A reduced relatively Hopfian cotorsion group is a Hopfian algebraically compact group.
\end{proposition}

\begin{proof} It follows from \cite[Theorem 55.5]{F1} that we can write $G=A\oplus C$, where $A$ is an algebraically compact torsion-free group and $C$ is a regulated cotorsion group.
	
Knowing that $A$ can be presented as $A=\prod_p A_p$ (see, e.g., \cite[\S 54]{F1}), where $A_p$ is the $p$-adic algebraically compact component of $A$, we derive that each $A_p$ is directly finite. So, it follows from this fact that its $p$-rank (equaling to that of $A/pA$) has to be finite, whence $A$ is Hopfian.
	
Moreover, the group $C$ also is presentable as $C=\prod_p C_p$, where $C_p$ is the $p$-adic module such that $C_p/T_p(C)$ is torsion-free divisible. Claim that each $C_p$ is a finite $p$-group. To show that, observe that if $T_p=T_p(C)$ is bounded, then $C_p=T_p$ and $C_p$ is finite by Corollary~\ref{prop3}.
	
Suppose now that $T_p$ is {\it not} bounded and let $B$ be a basic subgroup in $T_p$. Then, there exists a surjection
$\eta: B\to B\oplus X$ with non-zero kernel, where, as noticed above, $X$ is chosen to be a non-zero cyclic direct summand of $B$. Since $B$ and $B\oplus X$ are both torsion, we may extend $\eta$ to a surjection
$\overline{\eta}: C\to C\oplus X$ (e.g., \cite[\S 55, Exercise 1]{F1}), that forces $$C/H=(K/H)\oplus (L/H)$$ for some $\{0\}\neq H\leq C$, where $K/H\cong C$ and $L/H\cong X$. But, this makes no sense for relatively Hopfian groups.
	
So, both $A$ and $C$ are algebraically compact Hopfian groups, and because $C$ is fully invariant in $G=A\oplus C$, we infer that $G$ is also a Hopfian algebraically compact group, as stated.
\end{proof}

By analogy with \cite{AGW}, recall that a reduced mixed group $G$ with an infinite number of non-zero primary components $T_p(G)$ is said to be an {\it sp--group} if $G$ is a pure subgroup of the Cartesian product $\prod_p T_p(G)$. Note that the torsion subgroup $T(G)$ of $G$ is the corresponding direct sum $\bigoplus_p T_p(G)$, and the factor-group $G/T(G)$ is always divisible. For further details of sp-groups, we also refer to \cite{Ka} or \cite{KTT}.

\medskip

We now have the necessary machinery to state and prove the following.

\begin{proposition} Let $G$ be an sp-group of finite torsion-free rank, then $G$ is relatively Hopfian if, and only if, all the primary components $T_p(G)$ are relatively Hopfian.
\end{proposition}

\begin{proof} Since each $T_p(G)$ is a direct summand of the Cartesian product $\prod_q T_q(G)$, it is a direct summand of $G$ and hence Proposition \ref{prop3} guarantees that $T_p(G)$ is relatively Hopfian.
	
Conversely, assuming that each $T_p(G)$ is relatively Hopfian, then in virtue of Proposition \ref{allrel}, $T(G)$ is relatively Hopfian. Since $G/T(G)$ is torsion-free of finite rank, it follows from Proposition \ref{lemma220} that $G$ is relatively Hopfian, as formulated.
\end{proof}

We have seen in Proposition \ref{prop3} that the class of relatively Hopfian groups is closed under direct summands. Unfortunately, the situation in regard to the formation of direct sums is not so nice since the class $\mathcal {RH}$ will inherit the pathologies associated with Hopfian groups (compare also with \cite{Ch} and \cite{C1}).

\begin{example} Consider the group $G = A \oplus B$ exhibited by Corner in \cite[Example 2]{C2}. Here both $A,B$ are torsion-free Hopfian groups but their direct sum is not Hopfian, and since these groups are torsion-free, they are relatively Hopfian. However, $G$ cannot be relatively Hopfian, since it is not Hopfian.
\end{example}

As another example of the pathology associated with the formation of direct sums, consider \cite[Example 3]{C2} and the work of Goldsmith and V\'amos in \cite{GV}.

\begin{example}\label{ex3} For any positive integer $n \geq 2$, there is a Hopfian torsion-free group $G$ such that the direct sum $G^{(k)}$ is Hopfian for all integers $1 \leq k < n$, but $G^{(n)}$ is {\it not} Hopfian. Since the groups in question are all torsion-free, it is immediate that the Hopfian property may be replaced by the relative Hopfian property, so that  $G^{(k)}$ is relatively Hopfian for all integers $1 \leq k < n$, but $G^{(n)}$ is {\it not} relatively Hopfian and of course relative Hopficity is lost for all higher powers.
\end{example}
			
Recall a consequence of the discussions after Corollary~\ref{cor2.4} that we will use in our next examples: if $G$ is a reduced relatively Hopfian group, then each $p$-component $T_p(G)$ of $G$ is semi-standard.
	
\begin{proposition}\label{mix1}
Let $G$ be a reduced relatively Hopfian group with torsion subgroup $T = \bigoplus_p T_p$  and $G/T$ is Hopfian. Then, if either (i) $T$ is separable, or (ii) $p^{\omega}(T_p)$ is Hopfian for every prime $p$ (i.e., $T^1$ is Hopfian), then $G$ is Hopfian.
\end{proposition}
			
\begin{proof} One observes that, if $G\cong G/H$ for some non-zero subgroup $H$, then $F/H=T(G/H)$ for some $F \leq G$ and, therefore, $T + H\leq F$. So, we obtain that $$G/T\cong G/F \cong (G/T)/(F/T).$$
Since $G/T$ is Hopfian, we must have that $F=T = T+H$, i.e., $H\leq T$.

\medskip

We now distinguish the two critical cases:
				
\medskip
				
\textbf{Case (i):} $T$ is separable.

\medskip
				
Let $H = \bigoplus_p H_p$ be the standard primary decomposition of $H$; note that $H_p \leq T_p$, for each prime $p$, but some $H_p$ may be zero. Also, note as $G$ is relatively Hopfian, it follows, as indicated above, that each $p$-component $T_p$ is a semi-standard $p$-group.

Now, one finds that $T(G/H) = T/H$ since $H \leq T$, so taking $p$-primary components, we get that $T_p \cong T_p/H_p$. Since we are assuming $T$ is separable, each $T_p$ is a separable $p$-group too.
					
As $ H = \bigoplus_q H_q$, an arbitrary element, $x$ say, of $H$ may be expressed as a finite vector $x = (\dots, x_q, \dots)$ with each $x_q \in H_q$. For a fixed prime $q$, all $H_r$, with $r$ a prime different from $q$, are $q$-divisible and so $\mid x \mid^{q}_{T} = \mid x_q \mid^{q}_{T_q}$; indeed, since $T, T_q$ are pure in $G$, one has that $\mid x \mid^{q}_{G} = \mid x_q \mid^{q}_{T_q}$.
				
Fix, as we may, a prime $p$ such with $H_p$ non-zero. Then, $\min \mid H \mid^{p}_{G} = \min \mid H_p \mid^{p}_{T_p}$ is a finite integer since $T_p$ is separable. Furthermore, since $T_p/H_p \cong T_p$ and $T_p$ is semi-standard, we have from Lemma~\ref{mequalzero} that $\min \mid H \mid^{p}_G = 0$. It, thus, follows from Lemma~\ref{hzero} that $G$ is not relatively Hopfian. This contradiction forces $H = \{0\}$ and $G$ is then Hopfian, as required.
					
\medskip
				
\textbf{Case (ii):} $T^1$ is Hopfian.

\medskip
				
If $\min \mid H_p \mid^{p}_{G}$ is finite for some $p$, then an analogous argument to that in case (i) applies to show that $G$ is Hopfian. On the other hand, if $\min \mid H_p \mid^{p}_{G} \geq\omega$ for all $p$, then $H\leq T^1$ and hence $T^1\cong (T/H)^1 = T^1/H$. The hypothesis that $T^1$ is Hopfian then gives $H = \{0\}$ -- a contradiction. Consequently, in case (ii), the group $G$ is also Hopfian.
\end{proof}
			
As a consequence, we find:
		
\begin{corollary}\label{mix2}
If $G/t(G)$ is Hopfian and $t(G)$ is either separable, or $t(G)^1$ is Hopfian, then $G$ is relatively Hopfian if, and only if, it is Hopfian.
\end{corollary}

We can now establish the following result.
		
\begin{proposition} If $G$ is a reduced group such that $G/G^1$ is Hopfian and $t_p(G^1)=p^{\omega}t(G)$ is bounded for every $p$, then $G$ is relatively Hopfian.
\end{proposition}
		
\begin{proof} Assume, in a way of contradiction, that $G$ is {\it not} relatively Hopfian. Then, there is a subgroup $H$ of $G$ with $G/H = (K/H)\oplus(L/H)$, where $K/H\cong G$ and $L\neq H$. Now, $$G/L\cong K/H\cong G;$$ say $f: G\to G/L$ is an isomorphism. Thus, $f(G^1) = M/L$ for some $L\leq M\leq G$ so that $$f(G)/f(G^1) = (G/L)/(M/L)\cong G/M.$$
However, $f$ induces an isomorphism $$G/G^1\cong f(G)/f(G^1)$$ whence $G/M\cong G/G^1$. Now, if $\alpha$ is the canonical projection $G\twoheadrightarrow G/M$, then $$\alpha(G^1)\leq (\alpha(G))^1 = (G/M)^1 = \{\overline{0}\},$$
and hence $G^1\leq \Ker\,\alpha = M$.
			
Furthermore, $$G/G^1\cong G/M\cong (G/G^1)/(M/G^1)$$ and since, by hypothesis, $G/G^1$ is Hopfian, we must have that $M/G^1$ vanishes, i.e., $M = G^1$. But, it is obvious that $L/H$ is a pure subgroup of $G/H$, and as $H\leq L\leq M = G^1$, we obtain $$L/H\leq G^1/H = (G/H)^1.$$ Thus, as is well-known (see, e.g., [14, Exercise 6(c)]), $L/H$ is a divisible subgroup of $G/H$. This, however, is impossible unless $L/H$ is trivial, because $L/H\leq G^1/H$ is then also bounded. The contradiction that $L = H$ now guarantees the wanted result.
\end{proof}
		
Having at hand all the necessary ingredients, we are now ready to state and prove the following assertion.
		
\begin{proposition}
(i) If $G$ is a reduced relatively Hopfian group such that $G/t(G)$ is Hopfian, and if $G\cong G/H$ with $H\neq \{0\}$, then $H\leq t(G)$ and $t_p(H)$ is a bounded subgroup of $p^{\omega}(t_p(G))$ for any prime $p$.

\medskip
			
(ii) If $t_p(G)$ is semi-standard for every prime $p$, and if, for each non-trivial subgroup $H$ of $G$ satisfying $G/H\cong G$, we have $H\leq t(G)$ and $t_p(H)$ is a bounded subgroup of $p^{\omega}(t_p(G))$, then $G$ is relatively Hopfian.
\end{proposition}
		
\begin{proof} (i) Observe firstly that, it follows from an identical argument to that used in the proof Proposition~\ref{mix1} that, $H \leq t(G)$. Notice that the statement of part (i) is {\it not} vacuous: the hypothesis that $H\neq \{0\}$ is a guarantor that $G$ is {\it not} Hopfian, so that $t(G)^1$ is non-trivial as follows from Corollary~\ref{mix2}. Since $G$ is relatively Hopfian, $t_p(G)$ is, as observed previously, semi-standard for all primes $p$. Now, if it happens that $\min_G t_p(H)$ is finite for some $p$, then it automatically follows from Lemma~\ref{mequalzero} that $m = 0$. A further application of Lemma~\ref{hzero} then leads to a contradiction to the fact that $G$ is, by hypothesis, relatively Hopfian. So, we may assume that $\min_G t_p(H) \geq\omega$ and also that $t_p(H) \leq p^{\omega}G$ for all primes $p$.

\medskip
			
We now demonstrate that $t_p(H)$ is bounded. In fact, if $H$ is {\it not} bounded, then it has a basic subgroup, say $B$, with $H\neq B$. So, $H/B$ is divisible and thus we write $G/B = (K/B)\oplus(H/B)$ with $H/B\neq\{\overline{0}\}$ for some $K/B\cong G$; clearly, such a decomposition violates our hypothesis that $G$ is relatively Hopfian groups. Finally, in conclusion, $t_p(H)$ is bounded and $t_p(H) \leq  p^{\omega}(t_p(G)$, as required.

\medskip
			
(ii) Suppose, in a way of contradiction, that $G$ is {\it not} relatively Hopfian, so that there is a subgroup $M$ of $G$ with $G/M = (K/M)\oplus(L/M)$, where $K/M\cong G$ and $L\neq M$. Furthermore, $G/L\cong K/M\cong G$, so that, by hypothesis, $L\leq t_p(G)$ and $t_p(L)$ is a bounded subgroup of $p^{\omega}(t_p(G))$ for any prime $p$. So, $t_p(M)$ is also a bounded subgroup of $p^{\omega}(t_p(G))$.

\medskip
				
Now, the composition $p^\o t_p$ is itself a radical, being the composition of two radicals and since $t_p(M) \leq p^\o(t_p(G))$, we have by the standard properties of radicals that $$p^\o(t_p(G/M)) = p^\o(t_p(G))/t_p(M).$$ Hence,
			$$t_p(G)/t_p(M)/p^{\omega}(t_p(G)/t_p(M))=
			t_p(G)/t_p(M)/p^{\omega}(t_p(G))/t_p(M) \cong t_p(G)/p^{\omega}(t_p(G)).$$
			
Denote $G/M$ by $\overline{G}$ and note that $$t_p(G/M) = t_p(t(G/M)) = t_p((t(G)/M)) = t_p(G)/t_p(M).$$ It is then easy to check that $$t_p(\overline{G})/p^{\omega}(t_p(\overline{G}))\cong
			t_p(G)/p^{\omega}(t_p(G)),$$ and hence the corresponding Ulm invariants $f_n$ $(n<\omega)$ are equal. But, by a well-known result due to Kulikov (cf. \cite{F1,F}), we conclude that
			$$f_n(t_p(\overline{G})/p^{\omega}(t_p(\overline{G})))=
			f_n(t_p(\overline{G}))$$
and
			$$f_n(t_p(G)/p^{\omega}(t_p(G)))= f_n(t_p(G)),$$

\medskip

\noindent whence one detects that $f_n(t_p(\overline{G})) = f_n(t_p(G))$ for all finite $n$.

\medskip
			
Next, we then have
			$$f_n(t_p(\overline{G}))=f_n(t_p(K)/t_p(M)) + f_n(t_p(L)/t_p(M)) =
			f_n(t_p(G)) + f_n(t_p(L)/t_p(M)),$$

\medskip

\noindent where we have utilized the facts that $K/M\cong G$ and $M\leq L\leq t(G)$ to establish these two equalities.

\medskip
			
However, by hypothesis, the Ulm invariants are finite, so we can conclude that $f_n(t_p(L)/t_p(M)) = 0$ for all finite $n$. Now, seeing that the factor-group $t_p(L)/t_p(M)$ is bounded, this leads to the contradiction that $L/M$ is trivial. This establishes part (ii), thus completing the proof.
\end{proof}
		
\section{Open Problems}
		
We now end the article with two queries which may be rather difficult to resolve (see \cite{CD} as well).
		
\begin{problem}\label{p3} Describe the structure of those groups $G$ such that, for any subgroup $H$ of $G$, whenever $G$ is isomorphic to a {\bf proper} direct summand of the quotient-group $G/H$, then $H=\{0\}$.
\end{problem}
		
We assert that every Hopfian group possesses that property: in fact, suppose $G$ is a Hopfian group and let $\phi:G\to K$ be an isomorphism, $N$ a subgroup of $G$ and $G/N=K\oplus L$ a proper decomposition of the quotient-group $G/N$. If $\pi:G/N \to K$ is the canonical projection and $\gamma:G\to G/N$ is the canonical epimorphism, then the map $$\phi^{-1} \circ \pi \circ \gamma:G\to G$$ is the composition of three surjections, so it is surjective too. Since $G$ is Hopfian, it follows that this composition is injective, so that the map $\gamma$, which is the first of the existing maps from the right, must be injective as well. This fact automatically implies that $N=\{0\}$, and thus $G$ really possesses the property defined above, as asserted.
		
\medskip
		
We conjecture that this is just asking for a new description of the class of relatively Hopfian groups.
		
\begin{problem}\label{p4} Characterize up to an isomorphism those groups $G$ such that, for any subgroup $H$ of $G$, the isomorphism $G/H \cong G$ implies that $H$ is a {\bf proper} direct summand of $G$.
\end{problem}
		
Thereby, the trivial cases of $H=\{0\}$ and $H=G$ are obviously excluded. Besides, it can be elementarily seen by a direct verification that {\it all free groups of infinite rank possess this property}. We, thus, observe that it does specify a new class of groups different from the generalized Bassian ones as introduced in \cite{CDG2} (see also \cite{CDG1} as well as \cite{DG} and \cite{DK}, respectively, where it was independently proven that all the generalized Bassian group possess {\it finite} torsion-free rank).
		
\medskip
		
As a parallel, a good name for the class might be {\it generalized Hopfian}, because, certainly, under the presence of this definition, a generalized Bassian group will always be generalized Hopfian, but {\it not} the other way around.
		
\medskip
		
\noindent {\bf Funding:} The work of the first-named author, A.R. Chekhlov, is supported by the Ministry of Science and Higher Education of Russia (agreement No. 075-02-2024-1437). The work of the second-named author, P.V. Danchev, is partially supported by the Junta de Andaluc\'ia under Grant FQM 264.

\end{document}